\documentclass[11pt, reqno]{amsart}
\usepackage{amsmath, amssymb, graphicx, enumerate, amsthm, amscd, xspace}
\usepackage[pdfencoding=auto, psdextra]{hyperref}
\usepackage{paralist}

\usepackage{xcolor}
\hypersetup{
    colorlinks,
    linkcolor={red!50!black},
    citecolor={blue!50!black},
    urlcolor={blue!80!black}
}

\graphicspath{{./}{figures/}{../figures}}

\newtheorem{thm}{Theorem}[section]
\newtheorem{lem}[thm]{Lemma}

\theoremstyle{definition}
\newtheorem{defn}[thm]{Definition}
\newtheorem{defns}[thm]{Definitions}
\newtheorem{notn}[thm]{Notation}

\theoremstyle{remark}
\newtheorem{rem}[thm]{Remark}

\newcommand{\cT}{{{\mathcal{T}}}}

\newcommand{\fs}{{{\mathfrak{s}}}}
\newcommand{\fu}{{{\mathfrak{u}}}}

\newcommand{\R}{{{\mathbb{R}}}}

\newcommand{\hx}{{{\hat{x}}}}

\DeclareMathOperator{\den}{den}
\DeclareMathOperator{\trim}{trim}
\DeclareMathOperator{\ind}{index}

\newcommand{\nup}{\nu_{\fs \times \fu}}

\newcommand{\pA}{pseudo-Anosov\xspace}
\newcommand{\mpA}{measurable pseudo-Anosov\xspace}

\title{The dynamics of measurable pseudo-Anosov maps}
\date{July 2024}
\author{Philip Boyland}
\address{Department of Mathematics\\University of Florida\\372 Little
    Hall\\Gainesville\\ FL 32611-8105, USA}
\email{boyland@ufl.edu}
\author{Andr\'e de Carvalho}
\address{Departamento de Matem\'atica Aplicada\\ IME-USP\\ Rua Do Mat\~ao
    1010\\ Cidade Universit\'aria\\ 05508-090 S\~ao Paulo SP\\ Brazil}
\email{andre@ime.usp.br}
\author{Toby Hall}
\address{Department of Mathematical Sciences\\ University of Liverpool\\
    Liverpool L69 7ZL, UK}
\email{tobyhall@liverpool.ac.uk}

\thanks{The authors acknowledge the support of S\~ao Paulo Research Foundation
(FAPESP) grant number 2016/25053-8. This work is also partially supported by the
Simons Foundation Award No.\ 663281 granted to the Institute of Mathematics of the
Polish Academy of Sciences for the years 2021\,--\,2023.}

\begin{document}
\begin{abstract}
    We study the dynamics of  measurable pseudo-Anosov homeomorphisms of surfaces,
    a generalization of Thurston's pseudo-Anosov homeomorphisms.  A measurable
    pseudo-Anosov map has a transverse pair of full measure turbulations consisting of
    streamlines which are dense immersed lines: these turbulations are equipped
    with measures which are expanded and contracted uniformly by the homeomorphism.
    The turbulations need not have a good product structure anywhere, but have some
    local structure imposed by the existence of tartans: bundles of unstable and
    stable streamline segments which intersect regularly, and on whose
    intersections the product of the measures on the turbulations agrees with the
    ambient measure.

    We prove that measurable pseudo-Anosov maps are transitive, have dense periodic
    points, sensitive dependence on initial conditions, and are ergodic with respect to
    the ambient measure.

    Measurable pseudo-Anosovs maps were introduced in~\cite{ummpA}, where we
    constructed a parameterized family of non-conjugate examples on the sphere.
\end{abstract}

\maketitle

\section{Introduction}

One of the groundbreaking results proved by William Thurston in the late 1970s and
early 1980s was his Classification Theorem for isotopy classes of surface
homeomorphisms~\cite{Th, FLP}. Such a class contains a map of finite order, or it
contains a \pA map, or the surface can be cut into pieces which carry maps of these
two types. Pseudo-Anosov maps are defined as having a transverse pair of invariant
foliations, with finitely many pronged singularities, which carry a pair of holonomy
invariant measures, one expanding and the other contracting under the map.  In a
precise sense, pseudo-Anosov maps are simplest in their isotopy class, which implies,
among other things, that all maps in the class have infinitely many periodic orbits,
of infinitely many periods, and have positive topological entropy.

The Classification Theorem applies to surfaces of finite topological
type, that is, to compact surfaces from which finitely many points
have been removed.  Dynamicists have often put this theorem to use in
studying surface dynamics by finding periodic orbits, removing them
from the surface, and applying Thurston's theorem to the isotopy class
on the punctured surface thus obtained (\cite{bowen,Bdamster}).  It
then follows that the original map has at least as much dynamics as the
``hidden \pA map''.

Powerful though it is, there is a fundamental shortcoming to this approach when
studying the dynamics of parameterized families of maps: there are only countably
many isotopy classes of homeomorphisms on a surface of finite topological type, even
if we allow an arbitrary finite number of punctures. This means that it isn't
possible to capture all of the dynamical complexity of a whole parameterized family
using only canonical representatives from the isotopy classes of maps in the family
acting on the surface punctured at finite invariant sets. Thus we are faced with the
task of finding a new class of maps that retains many of the properties of \pA maps
but is rich enough to occur in every member of a parameterized family.

Our approach to this task was inspired by one of the standard approaches to
constructing \pA maps.  Beginning with a \pA isotopy class, a train track algorithm
such as that of Bestvina-Handel~\cite{BH,FM,Los} produces an expanding map on a
one-dimensional graph, its \emph{invariant train track}. Embedding the inverse
limit of the induced map acting on the track into the surface and then collapsing
down appropriate complementary regions yields the associated \pA map.  The holonomy
invariant measures on the invariant foliations are derived from a pair of measures
associated with the expanding map.

To work with a parameterized family, we can instead start with a family of piecewise
expanding maps on a \textit{fixed} graph and turn the same handle: take inverse
limits, embed into a surface and collapse complementary regions. We studied the
simplest such example, in which the graph is the interval and the family of maps is
the tent family. This led to the appropriate generalization of \pA maps considered
here, which we called \emph{measurable \pA maps} (see~\cite{prime, ummpA}).

\smallskip
Inverse limits of tent maps on the interval have been much studied by topologists
since the early 20th century and yield very complicated continua. For example, when
the parameter $s$ is such that critical orbit of the tent map $f_s$ of slope $s$ is
dense (a full measure, dense $G_\delta$ set of parameters), theorems of Bruin and of
Raines \cite{bruin, raines} imply that the inverse limit is nowhere locally the
product of a Cantor set and an interval. In addition, the collection of globally
regular (bi-dense, continuous injective images of $\R$) path components is meagre. In
contrast, and of critical importance here, with respect to the natural invariant
measure, the collection of globally regular path components has full measure.
Moreover, there is a pair of measures, one expanding by factor $s$ on path components
of the inverse limit, and the other contracting in the ``fiber'' direction by $1/s$;
these are built from Lebesgue measure on the interval and the unique absolutely
continuous invariant measure of the tent map respectively. Further, there are
positive measure sets, called \emph{boxes}, with a Fubini-like splitting connecting
these two measures to the natural, ambient, invariant measure \cite{typical}.

Embedding the inverse limit as an attractor in the disk using the Brown-Barge-Martin
procedure~\cite{BBM, usbbm} and then collapsing down complementary regions produces a
homeomorphism, $F_s$, of the sphere~\cite{prime}. The path components of the inverse
limit become unstable sets while the fiber direction collapses down to stable sets.
The ambient invariant measure on the inverse limit becomes an invariant Oxtoby-Ulam
measure on the sphere. The pair of measures on the inverse limit become unstable and
stable measures and the boxes become positive measure sets with a nice Fubini-like
structure called \emph{tartans}.  It is shown in~\cite{ummpA} that the resulting
\emph{unimodal measurable pseudo-Anosov} family~$\{F_s\}_{s\in(\sqrt 2,2]}$ has the
following properties:
\begin{itemize}
    \item For each $s\in(\sqrt 2,2]$, the map $F_s$ has topological
          entropy $\ln s$; in particular, no two members of the family are
          conjugate.
    \item Every member of the family is a \mpA map as defined here.
    \item For an uncountable dense set of values of~$s$, $F_s$ is a
          \emph{generalized pseudo-Anosov map} as defined in~\cite{eqgpa,gpA}.
    \item For a countable discrete set of parameter values, $F_s$ is a ``classical''
          pseudo-Anosov map with finitely many one-pronged singularities.
    \item However, for a full Lebesgue measure subset of parameters $s\in(\sqrt 2,2]$,
          $F_s$ is a ``genuine'' \mpA map in the sense that it is neither a classical \pA
          nor a generalized \pA map.
\end{itemize}

\smallskip
The definition of \mpA maps is general enough to include a large class of maps,
including, of course, the classical \pA\ maps introduced by Thurston, the generalized
\pA maps from~\cite{eqgpa,gpA}, and the family just constructed.  It has two
fundamental components. The first is that the transverse pair of invariant foliations
of a \pA map is generalized to a transverse pair of invariant \emph{turbulations} ---
decompositions of almost all of the surface into continuous injective images of $\R$,
called \emph{streamlines}, themselves equipped with 1-dimensional measures --- one of
which is uniformly expanded, and the other uniformly contracted, under the dynamics.
The second component is the ambient (2-dimensional) Oxtoby-Ulam invariant measure and
the \emph{tartans} which play the role of local charts. Each tartan has a product
structure using arcs from the stable and unstable streamlines, and the measures on
the streamlines are connected to the ambient invariant measure via a Fubini-like
hypothesis. Tartans form a family of positive measure sets whose union has full
measure.

\smallskip
This purpose of this article is to prove that this definition carries
dynamical content by showing that \mpA\ maps have the following properties:
\begin{itemize}
    \item They are transitive.
    \item The collection of their periodic orbits is dense.
    \item With a mild additional hypothesis, they are ergodic.
\end{itemize}

It is shown in~\cite{FLP} that Thurston's \pA maps have these properties and here we
extend them to the much more general class of \mpA maps.

\smallskip
As we have explained, the definition of measurable pseudo-Anosov maps was
predominantly guided by Thurston's definition and the examples derived from unimodal
family, but we were also influenced by the higher-dimensional holomorphic analogs
in~\cite{BLS,LM,Su}. In particular, the term `turbulation' is borrowed from comments
in~\cite{LM}.

\section{Turbulations, tartans, and \mpA maps}
\label{sec:gmpa-intro}

In this section we define measurable pseudo-Anosov maps, together with their
invariant measured turbulations and the tartans which give these turbulations local
structure. Throughout the paper, $\Sigma$ will denote a compact surface, and $\mu$
will denote an \emph{Oxtoby-Ulam} (or \emph{OU}) measure on~$\Sigma$: one which is
Borel, non-atomic, positive on open sets, and satisfies $\mu(\partial\Sigma)=0$.

The definitions are lengthy, and we start by motivating and summarising them. In
place of the measured foliations which are left invariant by a pseudo-Anosov map, the
invariant stable and unstable objects of a measurable pseudo-Anosov map are
\emph{measured turbulations}. Each turbulation is a union of disjoint immersed lines
(i.e.\ continuous injective images of~$\R$) called \emph{streamlines}, which take the
place of the leaves of pseudo-Anosov foliations. Each streamline is equipped with its
own OU \emph{stream measure}, which takes the place of the transverse measures on
pseudo-Anosov foliations. The measurable pseudo-Anosov map will send streamlines to
streamlines in each turbulation, expanding stream measures by a constant
factor~$\lambda$ in the \emph{unstable turbulation}, and contracting by a constant
factor~$1/\lambda$ in the \emph{stable turbulation}.

Figure~\ref{fig:turbulations} is a schematic depiction of some streamlines of a
typical transverse pair of turbulations in a subset of~$\Sigma$. The reader should
bear in mind that there may be backtracking (or `kinks') at every scale and in every
open subset of~$\Sigma$, so that the figure only shows the first level of complexity
of the turbulations. Note that, unlike pseudo-Anosov foliations which necessarily
have leaves which end at pronged singularities, every streamline of a turbulation is
a full immersed line: however the streamlines are not required to cover the whole
of~$\Sigma$, but only a full $\mu$-measure subset of it (so, for example, the
invariant foliations of a pseudo-Anosov map become invariant turbulations when we
throw away the singular leaves).

\begin{figure}[htbp]
    \begin{center}
        \includegraphics[width=0.8\textwidth]{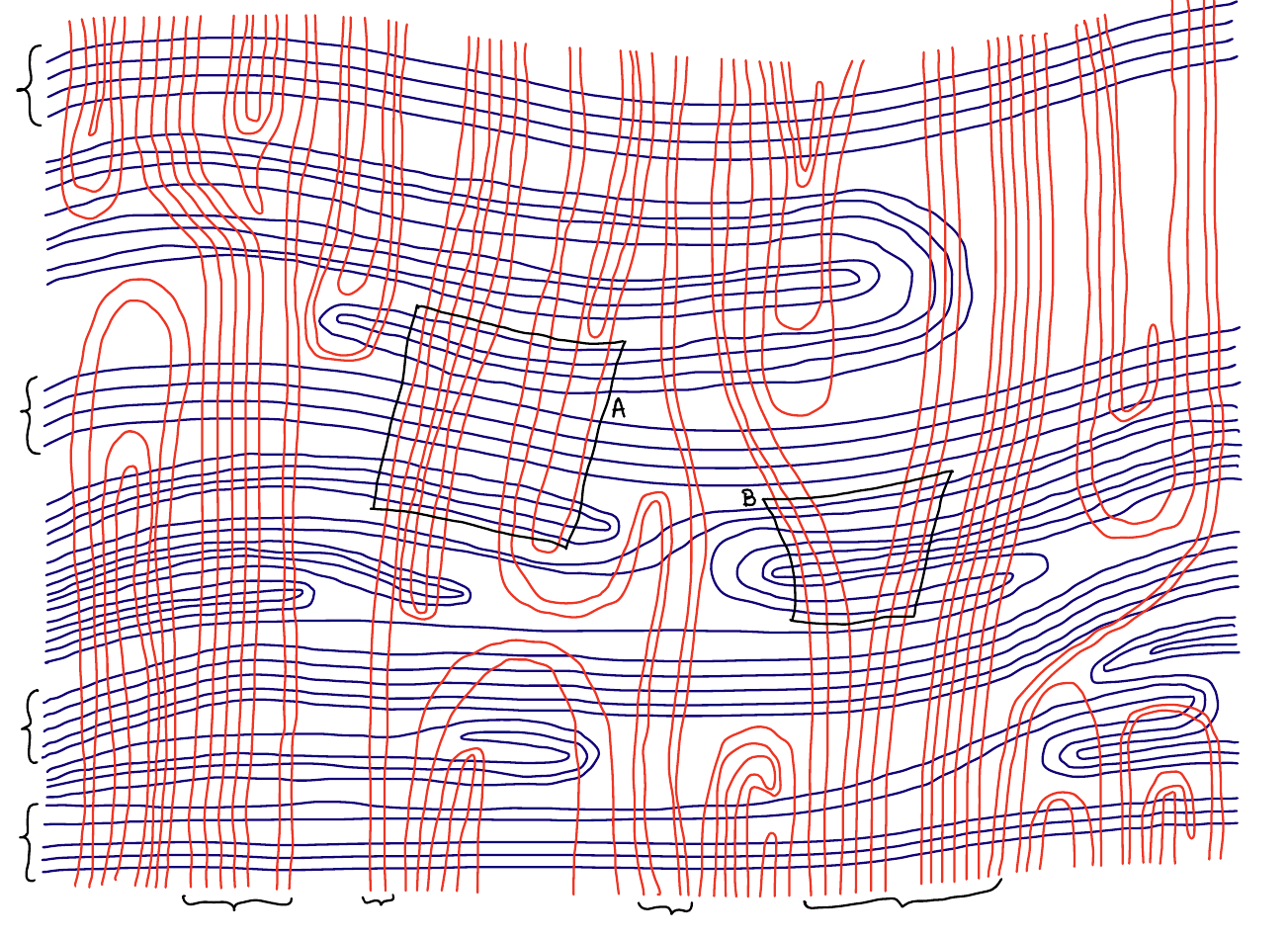}
    \end{center}
    \caption{Invariant turbulations of a Measurable pseudo-Anosov map (schematic)}
    \label{fig:turbulations}
\end{figure}

The turbulations may not define a product structure in any open subset. Nevertheless,
suitable collections of \emph{stream arcs} of the two turbulations do define a
product structure on the set of their intersections. On the largest scale in
Figure~\ref{fig:turbulations}, the stream arcs indicated by braces on the left and
bottom of the figure intersect regularly: each stream arc of one turbulation
intersects each stream arc of the other turbulation exactly once. We call such a
collection of stream arcs a \emph{tartan}, and the stream arcs which constitute it
\emph{fibers}: the tartan~$R$ endows the set~$R^\pitchfork$ of its intersection
points with a product structure, and we require the measure on~$R^\pitchfork$ induced
by the product of the stream measures on the fibers to agree with the ambient
measure~$\mu$: this is called \emph{compatibility} of~$R$. Provided that
$\mu(R^\pitchfork)>0$, the tartan has given structure to a significant subset
of~$\Sigma$. Tartans play the same structural r\^ole as charts and changes of charts
in the theory of foliations, and are reminiscent of product structures in the
differentiable dynamics of Axiom~A maps.

There are many other tartans in Figure~\ref{fig:turbulations}: for example, the
stream arcs in the regions labeled $A$ and $B$ --- or at least those stream arcs
which have been drawn --- are the fibers of tartans (some of whose intersections are
also intersections of the larger tartan). We say that the pair of turbulations is
\emph{full} if there is a countable collection of tartans which gives structure to a
full $\mu$-measure subset of~$\Sigma$ (the union of their intersection sets has full
measure), and in addition every non-empty open subset of~$\Sigma$ contains a
tartan~$R$ with $\mu(R^\pitchfork)>0$. The latter condition need not follow from the
former: additional regularity conditions which mean that it does are discussed in
Section~\ref{sec:pos-meas-tart}.

These are the essential ingredients for the definition of a measurable pseudo-Anosov
map: it should preserve a transverse full pair of measured turbulations, expanding
the streamlines of one and contracting the streamlines of the other by the same
uniform factor. There are some additional technical conditions, the most important of
which are:
\begin{itemize}
    \item by analogy with pseudo-Anosov foliations, every streamline of the invariant
          turbulations is dense in~$\Sigma$; and
    \item stream arcs of small measure have small diameter (with respect to some
          fixed metric on~$\Sigma$) --- this condition is called \emph{tameness}.
\end{itemize}

It turns out that holonomy invariance of the measures within tartans is a consequence
of compatibility (Lemma~\ref{lem:hol_inv}): if~$R$ is a compatible tartan and~$B$ is
a subset of $R^\pitchfork$ contained in a fiber of one of the turbulations, then the
stream measure of~$B$ is preserved as it is pushed into other fibers along fibers of
the second turbulation.

We now proceed with the formal definitions.

\begin{defns}[Measured turbulations, streamlines, dense streamlines, stream measure, transversality, stream arcs]
    \label{defn:turbulation}
    A \emph{measured turbulation $(\cT, \nu)$} on~$\Sigma$ is a partition of a full
    $\mu$-measure Borel subset of~$\Sigma$ into immersed lines called
    \emph{streamlines}, together with an OU measure $\nu_\ell$ on each
    streamline~$\ell$, which assigns finite measure to closed arcs in~$\ell$.
    \begin{itemize}
        \item We refer to the measures $\nu_\ell$ as \emph{stream measures} to
              distinguish them from the ambient measure~$\mu$ on~$\Sigma$.
        \item We say that the measured turbulation \emph{has dense streamlines} if
              every streamline is dense in~$\Sigma$.
        \item Two measured turbulations are \emph{transverse} if they are
              topologically transverse on a full $\mu$-measure subset of~$\Sigma$ (to be
              precise, there is a full measure subset of~$\Sigma$, every point~$x$ of which
              is contained in streamlines~$\ell$ and~$\ell'$ of the two turbulations, which
              intersect transversely at~$x$).
        \item Given distinct points $x$ and $y$ of a streamline~$\ell$ of~$\cT$, we
              write $[x,y]_\ell$, or just $[x,y]$ if the streamline is irrelevant or clear
              from the context, for the (unoriented) closed arc in~$\ell$ with
              endpoints~$x$ and~$y$; and $[x,y)_\ell$ and $(x,y)_\ell$ for the arcs
              obtained by omitting one or both endpoints of $[x,y]_\ell$. We refer to these
              as \emph{stream arcs} of the turbulation.
        \item The \emph{measure} of a stream arc is its stream measure.
    \end{itemize}
\end{defns}

\medskip

We impose a regularity condition on turbulations which requires that stream
arcs of small measure are small. Note that the following definition is independent
of the choice of metric on~$\Sigma$: since $\Sigma$ is compact, it is equivalent to
the topological condition that for every neighborhood~$N$ of the diagonal~$\{(x,
    x)\,:\,x\in\Sigma\}$ in~$\Sigma\times\Sigma$, there is some $\delta>0$ such that if
$[x, y]$ is a stream arc with $\nu([x, y])<\delta$, then $(x,y)\in N$. However the
metric formulation is usually easier to apply directly.

\begin{defn}[Tame turbulation] \label{defn:tame}
    Let~$d$ be any metric on~$\Sigma$ compatible with its topology. A measured turbulation $(\cT, \nu)$ on~$\Sigma$ is \emph{tame} if for
    every~$\epsilon>0$ there is some $\delta>0$ such that if $[x, y]$ is a stream
    arc with $\nu([x,y])<\delta$, then $d(x,y)<\epsilon$.
\end{defn}

\medskip

Let $(\cT^s, \nu^s)$ and $(\cT^u, \nu^u)$ be a transverse pair of
measured turbulations on~$\Sigma$. With a view to what is to come, we refer to
them as \emph{stable} and \emph{unstable} turbulations, and similarly apply the
adjectives stable and unstable to their streamlines, stream measures, stream
arcs, etc., although dynamics will not enter the picture until the very end of
this section. If~$x$ and $y$ lie in the same streamline of one or both
turbulations, we write, for example, $[x,y]_s$ and $[x,y]_u$ to denote the
stream arcs of the stable or unstable turbulation with $x$ and $y$ as
endpoints.


\begin{defns}[Tartan $R$, fibers, $R^\pitchfork$, positive measure tartan, oriented
        tartan]
    \label{defn:tartan} A \emph{tartan} $R=(R^s, R^u)$ consists of Borel
    subsets $R^s$ and $R^u$ of~$\Sigma$, which are disjoint unions of stable
    and unstable stream arcs respectively, having the following properties:
    \begin{enumerate}[(a)]
        \item Every arc of $R^s$ intersects every arc of $R^u$ exactly once,
              either transversely or at an endpoint of one or both arcs.

        \item There is a consistent orientation of the arcs of $R^s$ and $R^u$:
              that is, the arcs can be oriented so that every arc of $R^s$
              (respectively~$R^u$) crosses the arcs of $R^u$ (respectively~$R^s$) in the
              same order.

        \item The measures of the arcs of $R^s$ and $R^u$ are bounded above.

        \item There is an open topological disk $U\subseteq \Sigma$ which
              contains $R^s\cup R^u$.

    \end{enumerate}

    \medskip

    We refer to the arcs of $R^s$ and $R^u$ as the stable and unstable
    \emph{fibers} of~$R$, to distinguish them from other stream arcs. We write
    $R^\pitchfork$ for the set of \emph{tartan intersection points}: $R^\pitchfork
        =  R^s \cap R^u$. We say that~$R$ is a \emph{positive measure tartan} if
    $\mu(R^\pitchfork)>0$. An \emph{oriented tartan} is a tartan together with a
    consistent orientation of its stable and unstable fibers.
\end{defns}

Note that $R^s$ and $R^u$ are subsets of~$\Sigma$, but they decompose uniquely
as disjoint unions of stable/unstable stream arcs, so where necessary or
helpful can be viewed as the collection of their fibers.

\begin{notn}[$\fs(x)$, $\fu(x)$, $\fs\pitchfork\fu$, $E^\fs$, $E^\fu$, $\psi^{\fs,
                \fu}$, $\nu_\fs$ and $\nu_\fu$, product measure~$\nup$, holonomy maps
        $h_{\fs, \fs'}$, $h_{\fu, \fu'}$] \label{notn:tartan}
    \mbox{}

    Let~$R$ be a tartan; $x,y\in R^\pitchfork$; $\fs$ and $\fs'$ be stable
    fibers of~$R$; and $\fu$ and $\fu'$ be unstable fibers of~$R$. We write (see Figure~\ref{fig:defs}):
    \begin{itemize}
        \item $\nu^s_x$ and $\nu^u_x$ for the stream measures on the stable and
              unstable streamlines through~$x$.

        \item $\fs(x)$ and $\fu(x)$ for the stable and unstable fibers of~$R$
              containing~$x$.

        \item $\fs\pitchfork\fu = \fu\pitchfork\fs \in R^\pitchfork$ for the unique intersection point of $\fs$ and $\fu$.

        \item $E^\fs=\fs\cap R^u = \fs\cap R^\pitchfork$, and $E^\fu=\fu\cap R^s
                  = \fu\cap R^\pitchfork$;

        \item $\psi^{\fs, \fu}\colon E^\fs\times E^\fu \to R^\pitchfork$ for the
              bijection $(x,y)\mapsto \fu(x) \pitchfork \fs(y)$;

        \item $\nu_\fs$ and $\nu_\fu$ for the (restriction of the) stream measures
              on $E^\fs$ and $E^\fu$;

        \item $\nup$ for the product measure $\nu_\fs\times\nu_\fu$ on
              $E^\fs\times E^\fu$;

        \item $h_{\fs, \fs'}\colon E^\fs\to E^{\fs'}$ (respectively $h_{\fu,
                          \fu'}\colon E^\fu\to E^{\fu'}$) for the \emph{holonomy map} given by
              $x\mapsto \fu(x)
                  \pitchfork \fs'$ (respectively $y\mapsto \fs(y) \pitchfork \fu'$).
    \end{itemize}

\end{notn}

\begin{figure}[htbp]
    \begin{center}
        \includegraphics[width=0.55\textwidth]{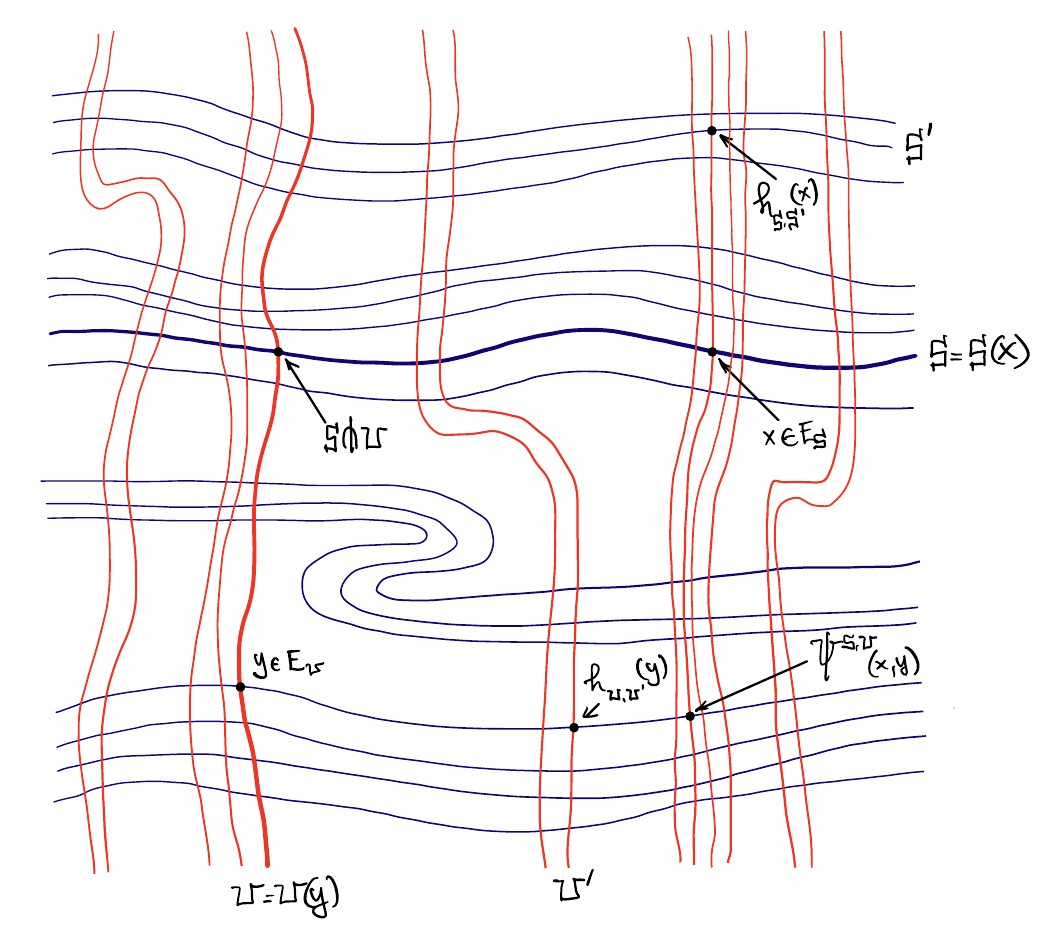}
    \end{center}
    \caption{Illustration of Notation~\ref{notn:tartan}}
    \label{fig:defs}
\end{figure}

Note that $\nu^s_x$ and $\nu^u_y$ are measures on entire
streamlines, whereas $\nu_\fs$ and $\nu_\fu$ are restrictions of such measures
to subsets of $R^\pitchfork$.

\begin{defn}[Compatible tartan]
    We say that the tartan $R=(R^s, R^u)$ is \emph{compatible} (with the
    ambient measure~$\mu$) if, for all stable and unstable fibers~$\fs$
    and~$\fu$, the bijection $\psi^{\fs, \fu}\colon E^\fs\times E^\fu \to
        R^\pitchfork$ is bi-measurable, and $\psi^{\fs, \fu}_* \nup =
        \mu|_{R^\pitchfork}$.
\end{defn}

If~$R$ is compatible and $A\subset R^\pitchfork$ is Borel, then using Fubini's
theorem in the product $E^\fs \times E^\fu$ and pushing forward
to~$R^\pitchfork$ we have, for any stable and unstable fibers~$\fs$ and~$\fu$,
\begin{equation}
    \label{eq:int_formula}
    \mu(A) = \int_{E^\fs} \nu_\fu(A^\fu(x))\,d\nu_\fs(x) = \int_{E^\fu}
    \nu_\fs(A^\fs(y))\,d\nu_\fu(y),
\end{equation}
where $A^\fu(x) = \{y\in E^\fu\,:\, \fu(x)\pitchfork\fs(y) \in A\}$, and
\mbox{$A^\fs(y) = \{x\in E^\fs\,:\, \fu(x)\pitchfork\fs(y) \in A\}$}.

\begin{lem}[Holonomy invariance] \label{lem:hol_inv}
    Let~$R$ be a compatible positive measure tartan, $\fs$ and $\fs'$ be stable
    fibers of~$R$, and $B \subset E^\fs$ be
    $\nu_\fs$-measurable. Then $B' = h_{\fs, \fs'}(B)$ is
    $\nu_{\fs'}$-measurable, and $\nu_{\fs'}(B') =
        \nu_\fs(B)$. The analogous statement holds for holonomies~$h_{\fu,
                \fu'}$.
\end{lem}

\begin{proof}
    Let~$\fu$ be any unstable fiber. Then $\nu_\fu(E^\fu)>0$ by~\eqref{eq:int_formula} with $A=R^\pitchfork$.

    Since $x$ and $h_{\fs, \fs'}(x)$ lie on the same unstable fiber~$\fu(x)$ for
    all $x\in B$, we have $\psi^{\fs, \fu}(B\times E^\fu) = \psi^{\fs',
            \fu}(B'\times E^\fu)$. Then $\nu_\fs$-measurability of~$B$ implies
    $\mu$-measurability of~$\psi^{\fs, \fu}(B\times E^\fu) = \psi^{\fs',
            \fu}(B'\times E^\fu)$, and hence $\nu_{\fs'}$-measurability of~$B'$.

    Since $\mu(\psi^{\fs, \fu}(B\times E^\fu)) = \nu_\fs(B)\nu_\fu(E^\fu)$ and
    $\mu(\psi^{\fs', \fu}(B'\times E^\fu)) = \nu_{\fs'}(B')\nu_\fu(E^\fu)$, the
    result follows.
\end{proof}

It is useful to have a version of~\eqref{eq:int_formula} reflecting the fact
that, in compatible tartans, the stream measures are the disintegration of the
ambient measure onto fibers:
\begin{equation}
    \label{eq:int_formula2}
    \mu(A) = \int_{E^\fs} \nu_{\fu(x)}(A\cap \fu(x))\,d\nu_\fs(x) =
    \int_{E^\fu} \nu_{\fs(y)}(A \cap \fs(y))\,d\nu_\fu(y).
\end{equation}
This follows because
\begin{eqnarray*}
    \nu_{\fs(y)}(A \cap \fs(y)) &=& \nu_\fs(h_{\fs(y), \fs}(A\cap \fs(y))) \\
    &=& \nu_\fs(\{\fu(x)\pitchfork \fs\,:\, x\in A \cap \fs(y)\}) \\
    &=& \nu_\fs(\{x'\in E^\fs\,:\, \fu(x')\pitchfork\fs(y)\in A\}) \\
    &=& \nu_\fs(A^\fs(y)),
\end{eqnarray*}
and likewise $\nu_{\fu(x)}(A\cap \fu(x)) = \nu_\fu(A^\fu(x))$ for any
$x\in E^\fs$.

\begin{defns}[Stable and unstable widths $w^s(R)$, $w^u(R)$] Taking
    $B=\fs\cap R^\pitchfork = E^\fs$ in Lemma~\ref{lem:hol_inv}, we see that if~$R$
    is compatible and positive measure then $\nu_\fs(E^\fs)>0$ is independent of
    the stable fiber~$\fs$; and similarly $\nu_\fu(E^\fu)>0$ is independent of the
    unstable fiber~$\fu$. We write $w^s(R)$ and $w^u(R)$ for these quantities (the
    stable and unstable \emph{widths} of~$R$), and note that $\mu(R^\pitchfork) =
        w^s(R)\,w^u(R)$ by~\eqref{eq:int_formula}.
\end{defns}

In order that the ambient measure~$\mu$ is globally compatible with the stream
measures, it is clearly necessary for a full measure subset of~$\Sigma$ to be
covered by intersections of positive measure compatible tartans: this is
condition~(a) in the following definition. However, a corresponding topological
condition, given by part~(b) of the definition, is also useful: it follows
from~(a) if an additional regularity condition is imposed on tartans or on the
turbulations. We return to this point in Section~\ref{sec:pos-meas-tart}.

\begin{defn}[Full turbulations] \label{defn:full-turb}
    We say that the transverse pair $(\cT^s, \nu^s)$ and $(\cT^u, \nu^u)$ is
    \emph{full} if
    \begin{enumerate}[(a)]
        \item There is a countable collection~$R_i$ of (positive measure) compatible tartans with $\mu\left(\bigcup_i R_i^\pitchfork\right)=1$; and

        \item for every non-empty open subset~$U$ of~$\Sigma$, there is a
              positive measure compatible tartan $R=(R^s, R^u)$ with $R^s\cup
                  R^u\subset U$.
    \end{enumerate}
\end{defn}

We are finally in a position to be able to define measurable pseudo-Anosov maps.

\begin{defns}[Image turbulation, measurable pseudo-Anosov turbulations, measurable
        pseudo-Anosov map, dilatation]
    \label{defn:mpa}
    If $F\colon \Sigma\to \Sigma$ is a $\mu$-preserving homeomorphism, we write
    $F(\cT,\nu)$ for the measured turbulation whose streamlines are
    $\{F(\ell)\,:\,\ell\in\cT\}$, with measures $\nu_{F(\ell)} =
        F_*(\nu_\ell)$.

    A pair $(\cT^s, \nu^s), (\cT^u, \nu^u)$ of measured turbulations
    on~$\Sigma$ are \emph{measurable pseudo-Anosov turbulations} if they are
    transverse, tame, full, and have dense streamlines.

    A $\mu$-preserving homeomorphism $F\colon \Sigma\to \Sigma$ is \emph{measurable
        pseudo-Anosov} if there is a pair $(\cT^s, \nu^s), (\cT^u, \nu^u)$ of
    measurable pseudo-Anosov turbulations and a number~$\lambda>1$, called the
    \emph{dilatation} of~$F$, such that $F(\cT^s, \nu^s) = (\cT^s,
        \lambda\nu^s)$ and $F(\cT^u, \nu^u) = (\cT^u, \lambda^{-1}\nu^u)$.
\end{defns}

Note that the positions of $\lambda$ and $\lambda^{-1}$ here are not errors: they
differ from those familiar in the definition of pseudo-Anosov maps since we are
using measures along streamlines rather than transverse to them.

\begin{rem}
    If $F\colon\Sigma\to\Sigma$ is a \mpA map, then so is any homeomorphism
    topologically conjugate to it.
\end{rem}

\section{Subtartans}

This short section describes some useful constructions of subtartans.

\begin{defn}[Subtartan] \label{defn:subtartan}
    Let~$R=(R^s, R^u)$ and $S=(S^s, S^u)$ be tartans. We say that~$S$ is a
    \emph{subtartan} of~$R$ if every stable (respectively unstable) fiber of~$S$ is
    contained in a stable (respectively unstable) fiber of~$R$.
\end{defn}

The simplest way to obtain a subtartan of~$R$ is just to discard some of its
fibers: if $S^s$ and $S^u$ are Borel subsets of $R^s$ and $R^u$ which are
unions of fibers of~$R$, then $S$ is a subtartan of~$R$, which is clearly
compatible if~$R$ is.

Another construction is to take a rectangle bounded by segments of two stable
and two unstable fibers of~$R$, and to form a subtartan whose fibers are the
intersections of fibers of~$R$ with this rectangle.

\begin{defn}[$R(x,y)$] \label{defn:subtartan-1}
    Let $R$ be a compatible tartan, and let $x,y\in R^\pitchfork$ lie on different
    stable and unstable fibers of~$R$. Let $D$ be the closed disk bounded by the
    arcs $[x, \fs(x)\pitchfork\fu(y)]_s$, $[\fs(x)\pitchfork\fu(y), y]_u$, $[y,
                \fs(y)\pitchfork\fu(x)]_s$, and $[\fs(y)\pitchfork\fu(x), x]_u$. We define
    $R(x,y)$ to be the subtartan of~$R$ whose stable and unstable fibers are the
    intersections of the stable and unstable fibers of~$R$ with~$D$.

    It is easily seen that $R(x,y)$ is compatible, and that
    $w^s(R(x,y))=\nu_{\fs(x)}([x, \fs(x)\pitchfork\fu(y)]_s)$ and
    $w^u(R(x,y))=\nu_{\fu(x)}([x, \fs(y)\pitchfork \fu(x)]_u)$.
\end{defn}

We can also construct subtartans which consist of all the fibers of~$R$ which
intersect given stable and unstable fibers in prescribed points. We start by
defining how to \emph{trim} a tartan to remove `loose' segments of fibers.

\begin{defn}[$\trim(R)$] \label{defn:trim}
    A tartan~$R$ can be \emph{trimmed} to yield a subtartan $\trim(R)$, by throwing
    away the ends of its fibers which contain no intersections. Each stable fiber
    $\fs$ of~$R$ yields a stable fiber $\fs'$ of $\trim(R)$, which is the minimal
    (open, closed, or half-open) subarc of~$\fs$ which contains $E^\fs$, and
    analogously for unstable fibers.
\end{defn}

\begin{defn}[$R(I^s, I^u)$]
    Let~$R = (R^s, R^u)$  be a positive measure oriented compatible tartan, and let
    $\fs$ and $\fu$ be stable and unstable fibers of~$R$.

    For each $x\in E^\fs = \fs\cap R^\pitchfork$, we denote by $\iota_x$ the
    initial (with respect to the given orientation of~$\fs$) stream arc of $\fs$
    with terminal endpoint~$x$ (which may be open or closed at its initial point,
    according as~$\fs$ is). Given a subinterval $I^s \subset [0,1]$, define
    $A^s(I^s)\subset E^\fs$ by
    \[
        A^s(I^s) = \{x\in E^\fs\,:\,
        \nu_\fs(\iota_x \cap R^\pitchfork) / w^s(R) \in I^s\}.
    \]

    Likewise, if $I^u\subset [0,1]$ is a subinterval, we define a corresponding
    subset $A^u(I^u)$ of~$E^\fu$. We then define $R(I^s, I^u)$ by trimming the
    subtartan of~$R$ whose stable fibers are $\fs(y)$ for $y\in A^u(I^u)$, and
    whose unstable fibers are $\fu(x)$ for $x\in A^s(I^s)$.

    The definition is independent of the choice of fibers $\fs$ and $\fu$
    by Lemma~\ref{lem:hol_inv}.
\end{defn}

For example, $S = R([1/3, 2/3], [1/3, 2/3])$ is the `middle thirds' subtartan
of~$R$, with $\mu(S^\pitchfork) = \mu(R^\pitchfork)/9$ (in contrast to the
middle thirds Cantor set, here we \emph{retain} the middle thirds of the stable
and unstable fibers).

\section{Density}

Let~$\alpha$ be an oriented stable or unstable stream arc with stream
measure~$\nu$. Then there is a \emph{stream metric~$d_\alpha$} on~$\alpha$
defined by $d_\alpha(x,y) = \nu([x, y]_\alpha)$, with respect to which~$\alpha$
is measure-preserving and orientation-preserving isometric to a finite
interval~$I$ in~$\R$ with Lebesgue measure --- in one dimension, metrics and OU
measures are equivalent. In particular, we can apply the Lebesgue Density Theorem
to stream arcs. In this section we present some consequences which will be used
later. In Lemma~\ref{lem:density-pts} we show that, for a compatible tartan~$R$,
the set $\den_1(R)$ of points of~$R^\pitchfork$ which are (two-sided) density
points of $R^\pitchfork$ along both stable and unstable fibers has full measure:
moreover, the same is true for $\den_2(R)$, the set of points which are density
points of $\den_1(R)$ along both fibers. Lemma~\ref{lem:density_intervals}
concerns points~$x$ for which all small stream arcs with endpoint~$x$ contain a
high density of points of~$R^\pitchfork$.

\begin{notn}
    Let $\varphi\colon I\to \alpha$ be such a measure-preserving and
    orientation-preserving isometry from an interval in~$\R$ to an oriented stream
    arc. Pick some $x=\varphi(\hx)\in\alpha$. If $a\in\R$ is small enough that
    $\hx+a\in I$, then we denote by $x + a$ the point $\varphi(\hx + a)$ of
    $\alpha$.

    When we write, for example, $[x, x+a]_s$ (respectively $[x, x+a]_u$), we mean
    the stream arc with endpoints $x$ and $x+a$ on the stable (respectively
    unstable) streamline through~$x$: by construction, these arcs have stream
    measure~$|a|$. In practice the orientation of~$\alpha$ will be irrelevant,
    since the definitions which follow are invariant under $a\mapsto -a$.
\end{notn}

\begin{defn}[$\den_k(R)$]
    Let~$R$ be a compatible tartan. We define $\den_k(R)\subset R^\pitchfork$, the
    \emph{set of level~$k$ density points of~$R$}, inductively for $k\ge 0$ by
    $\den_0(R) = R^\pitchfork$ and $\den_k(R) = \den_k^s(R) \cap \den_k^u(R)$,
    where
    \[
        \den_k^{s/u}(R) = \{x\in \den_{k-1}(R) \,:\,
        \lim_{a\to 0}
        \frac{\nu^{s/u}_x([x, x+a]_{s/u} \cap \den_{k-1}(R))}{|a|}
        =1
        \}.
    \]
\end{defn}

That is, an element~$x$ of~$\den_k(R)$ is an element of~$R^\pitchfork$ which is a
(2-sided) density point of~$\den_{k-1}(R)$ along both $\fs(x)$ and $\fu(x)$.

\begin{lem} \label{lem:density-pts}
    Let~$R$ be a compatible tartan. Then $\mu(\den_k(R)) = \mu(R^\pitchfork)$ for
    all~$k\ge0$.
\end{lem}

\begin{proof}
    The proof is by induction on~$k$, with the case $k=0$ vacuous. We can assume
    that $\mu(R^\pitchfork)>0$.

    Let $\fu$ be any unstable fiber of~$R$, and define
    \[
        T = \{y\in E^\fu\,:\,
        \nu_{\fs(y)}(\den_{k-1}(R)\cap\fs(y)) = w^s(R)\}.
    \]
    Since $\mu(R^\pitchfork) - \mu(\den_{k-1}(R)) = 0$ by the inductive hypothesis,
    we have, by~\eqref{eq:int_formula2},
    \[
        \int_{E^\fu} \left(w^s(R) - \nu_{\fs(y)}(\den_{k-1}(R)\cap \fs(y))\right)
        d\nu_u(y) = 0,
    \]
    and so $\nu_\fu(T)=\nu_\fu(E^\fu)=w^u(R)$.

    Let $y\in T$. Since $\den_{k-1}^s(R)\cap \fs(y) =
        \den_{k-1}^s(R)\cap E^{\fs(y)}$, we have
    \[
        \den^s_k(R)\cap \fs(y) = \{
        x\in\den_{k-1}(R)\cap E^{\fs(y)} \,:\,
        \lim_{a\to 0}
        \frac{\nu^s_x([x, x+a]_s \cap (\den_{k-1}(R)\cap E^{\fs(y)}))}{|a|}
        =1
        \}.
    \]
    Therefore, by the Lebesgue Density Theorem we have, writing $\ell$ for the
    streamline containing $\fs(y)$,
    \begin{eqnarray*}
        \nu_{\fs(y)}(\den_k^s(R)\cap \fs(y))
        &=& \nu_\ell(\den_k^s(R)\cap \fs(y))
        = \nu_\ell(\den_{k-1}(R)\cap E^{\fs(y)})\\
        &=& \nu_{\fs(y)}(\den_{k-1}(R)\cap\fs(y)) = w^s(R).
    \end{eqnarray*}

    It follows from~\eqref{eq:int_formula2} that
    \[
        \mu(\den_k^s(R)) = \int_T \nu_{\fs(y)}(\den_k^s(R) \cap\fs(y)) d\nu_\fu(y)
        = w^u(R)w^s(R) = \mu(R^\pitchfork).
    \]

    The same argument shows that $\mu(\den_k^u(R)) = \mu(R^\pitchfork)$, and the
    result follows.
\end{proof}

We now restrict to level~1 density points, and consider the set of points~$x\in
    R^\pitchfork$ for which the density of intersections in all small stable and
unstable stream arcs with endpoint~$x$ is at least some prescribed value.

\begin{defn}[$X_{\delta, \eta}(R)$]
    \label{def:density_intervals}
    Let~$R$ be a compatible tartan. Given $\delta\in(0,1)$ and $\eta>0$, we define
    \[
        X_{\delta, \eta}^{s/u}(R) =
        \{
        x\in R^\pitchfork\,:\,
        \nu_x^{s/u}([x, x+a]_{s/u}\cap R^\pitchfork) \ge
        (1-\delta)|a| \text{ for all }a\in (-\eta, \eta)
        \},
    \]
    and
    \[
        X_{\delta, \eta}(R) =
        X_{\delta, \eta}^{s}(R) \cap X_{\delta, \eta}^{u}(R).
    \]
\end{defn}

\begin{lem}\label{lem:density_intervals}
    Let~$R$ be a compatible tartan and $\delta\in(0,1)$. Then
    \[
        \lim_{\eta\to 0}\mu(X_{\delta, \eta}(R)) = \mu(R^\pitchfork).
    \]
\end{lem}

\begin{proof}
    Note that if $\eta'<\eta$ then $X_{\delta, \eta}(R) \subset X_{\delta,
                \eta'}(R)$. Let
    \[
        X_\delta(R) = \bigcup_{\eta>0} X_{\delta, \eta}(R) = \bigcup_{k\ge 1}
        X_{\delta, 1/k}(R).
    \]
    By continuity of measure, it suffices to show that $\mu(X_\delta(R)) =
        \mu(R^\pitchfork)$. This follows from Lemma~\ref{lem:density-pts} since
    $\den_1(R)\subset X_\delta(R)$, as can be easily shown.
\end{proof}

\section{Positive measure tartans in open sets} \label{sec:pos-meas-tart}

Recall (Definition~\ref{defn:full-turb}) that for a transverse pair of turbulations
to be \emph{full}, there are two requirements: first, that there is a countable
collection of compatible tartans whose intersections cover a full measure subset
of~$\Sigma$; and second, that every non-empty open set contains a positive measure
compatible tartan. We now discuss conditions which, combined with the first of
these requirements, imply the second.

Let~$U$ be a non-empty open subset of~$\Sigma$. Since $\mu$ is OU we have
$\mu(U)>0$, and hence, by the first requirement, there is a tartan~$R$ with
$\mu(R^\pitchfork \cap U) > 0$. By Lemma~\ref{lem:density-pts} it follows that
$\mu(\den_2(R) \cap U) > 0$. Let~$x\in\den_2(R)\cap U$.

For each~$i\ge 1$, pick~$a_i>0$ such that $[x, x+a_i]_s$ has diameter less
than~$1/i$. By definition of~$\den_2(R)$, we can pick $y_i\in (x, x+a_i]_s \cap
    \den_1(R)$: moreover, we have $\nu_{\fs(x)}([x, y_i]_s \cap R^\pitchfork) \ge
    \nu_{\fs(x)}([x, y_i]_s \cap \den_1(R)) > 0$.

Similarly, given~$i$, for each~$j\ge 1$ we pick~$b_{i,j}>0$ such that $[y_i, y_i +
            b_{i, j}]_u$ has diameter less than~$1/j$. Since~$y_i\in \den_1(R)$, we can pick
$z_{i,j}\in(y_i, y_i+b_{i,j}]_u\cap R^\pitchfork$; and $\nu_{\fu(y_i)}([y_i, z_{i,
            j}]_u \cap R^\pitchfork) > 0$.

For each~$i$, the diameter of $[x, \fs(z_{i,j})\pitchfork \fu(x)]_u$ tends to zero
as $j\to\infty$. For, given~$\epsilon>0$, let~$\delta>0$ be such that $[x,
            x+\delta]_u$ has diameter less than~$\epsilon$, and let $w\in(x, x+\delta]_u\cap
    R^\pitchfork$ (such a point exists since $x\in\den_2^u(R)$). Let~$j$ be large
enough that $1/j$ is less than the diameter of $[y_i, \fu(y_i)\pitchfork
            \fs(w)]_u$: then $z_{i,j}\in [y_i, \fu(y_i)\pitchfork \fs(w)]_u$, so that
$\fs(z_{i,j})\pitchfork \fu(x) \in[x, x+\delta]_u$ by
Definitions~\ref{defn:tartan}(b).

Each $i\ge 1$ and $j\ge 1$ yields a compatible subtartan $R(x, z_{i,j})$ of~$R$,
with measure $\nu_{\fs(x)}([x, y_i]_s \cap R^\pitchfork)\, \nu_{\fu(y_i)}([y_i,
        z_{i, j}]_u \cap R^\pitchfork) > 0$. If~$i$ and~$j$ are sufficiently large then the
arcs $[x, y_i]_s$, $[y_i, z_{i,j}]_u$, and $[x, \fs(z_{i,j})\pitchfork\fu(x)]_u$
are contained in~$U$. Thus it suffices to show that we can choose~$i$ and~$j$ so
that the remaining bounding arc $[z_{i, j},\fs(z_{i,j})\pitchfork\fu(x)]_s$ of
$R(x, z_{i, j})$ is also contained in~$U$.  This is the step that requires an
additional regularity condition on tartans or on the measured turbulations.

The simplest such condition just says that the result we need is true:
\begin{defn}[Regular tartan] \label{defn:regular-tartan}
    We say that a tartan~$R$ is \emph{regular} if for all $x\in R^\pitchfork$ and
    all neighborhoods~$U$ of~$x$, there is some~$\delta>0$ such that if $y\in
        \fs(x)\cap R^\pitchfork$ and $z\in \fu(x)\cap R^\pitchfork$ with
    $\nu_{\fs(x)}([x,y]_s)<\delta$ and $\nu_{\fu(x)}([x,z]_u)<\delta$, then
    $R(y,z)\subseteq U$ (that is, all of the fibers of $R(y,z)$ are contained
    in~$U$).
\end{defn}

A second condition which is perhaps more natural, in that it relates metric to stream measure within tartans, is:

\begin{defn}[Bi-Lipschitz regular tartan]\label{defn:biLipschitz-reg-tartan}
    We say that a tartan~$R$ is \emph{(stable) bi-Lipschitz regular} if there is a
    metric~$d$ on~$\Sigma$ compatible with its topology and a constant~$K\ge 1$
    such that whenever $x$ and $y$ lie on the same stable fiber of~$R$, we have
    \[
        \frac{1}{K} \, d(x,y) \le \nu_s([x,y]_s) \le K d(x,y).
    \]
\end{defn}

\medskip

A final condition is of interest because it can be expressed solely in terms of a
single turbulation.

\begin{defn}[Partial flowboxes]\label{defn:partial-flowbox}
    A turbulation $\cT$ \emph{has partial flowboxes} if whenever
    \begin{itemize}
        \item $x$ and $y$ lie on the same streamline~$\ell$ of~$\cT$, and
        \item there are arcs $\alpha_x$ and $\alpha_y$ in~$\Sigma$, transverse
              to~$\ell$ at $x$ and~$y$, which contain transverse intersections with
              streamlines of~$\cT$ arbitrarily close to~$x$ and~$y$ on the same side
              of~$\ell$,
    \end{itemize}
    then every neighborhood~$U$ of the stream arc $[x,y]$ contains other stream
    arcs with endpoints on $\alpha_x$ and $\alpha_y$.
\end{defn}

\begin{lem}[Conditions for positive measure tartans in open sets]
    \label{lem:pos-cond} Let $(\cT^s, \nu^s)$ and $(\cT^u, \nu^u)$ be a transverse
    pair of measured turbulations which satisfy Definition~\ref{defn:full-turb}(a),
    namely that there is a countable collection~$R_i$ of compatible tartans with
    $\mu(\bigcup_i R_i^\pitchfork)=1$. Then the pair also satisfies
    Definition~\ref{defn:full-turb}(b), and so is full, if any of the following
    holds:
    \begin{enumerate}[(a)]
        \item The tartans in Definition~\ref{defn:full-turb}(a) can be chosen to
              be regular; or

        \item The tartans in Definition~\ref{defn:full-turb}(a) can be chosen to
              be bi-Lipschitz regular; or

        \item The turbulations have partial flowboxes.
    \end{enumerate}
\end{lem}

\begin{proof}
    Let~$U$ be a non-empty open subset of~$\Sigma$. As in the introduction to this
    section, we can find a tartan~$R$ and a point $x\in R^\pitchfork \cap U$; and
    points $y\in\fs(x)$, $z\in\fu(y)$, and $w=\fs(y)\pitchfork\fu(x)$ arbitrarily
    close to~$x$ such that the subtartan $R(x,z)$ is compatible and has positive
    measure, and moreover $[x,y]_s$, $[x,w]_u$, and $[y,z]_u$ are contained in~$U$.
    Thus it is only required to show that $y$, $z$, and $w$ can be chosen close
    enough to~$x$ that $[z,w]_s$ is also contained in~$U$.

    This is clear if~$R$ is regular: if~$i$ and~$j$ are sufficiently large, then
    the points~$y$ and~$w$ satisfy $\nu_{\fs(x)}([x,y]_s)<\delta$ and
    $\nu_{\fu(x)}([x,w]_u)<\delta$, where~$\delta>0$ is provided by
    Definition~\ref{defn:regular-tartan}.

    If~$R$ is stable bi-Lipschitz regular, then let~$d$ be the metric and~$K$ be
    the Lipschitz constant provided by
    Definition~\ref{defn:biLipschitz-reg-tartan}. Let~$\epsilon>0$ be small enough
    that $B_d(x, 2\epsilon)\subset U$, and let $\delta = \epsilon/K^2$. Pick $y$,
    $z$, and $w$ close enough to~$x$ that $d(z,w)<\delta$ and $d(x,z)<\epsilon$.
    Then $\nu_s([z,w]_s) < K\delta$, so that $\nu_s([z,v]_s) < K \delta$ for all
    $v\in [z,w]_s$, and hence $d(z,v) < K^2\delta = \epsilon$ for all $v\in
        [z,w]_s$. Hence every $v\in [z,w]_s$ has $d(x,v) \le d(x,z)+d(z,v) <
        2\epsilon$, and $[z,w]_s$ is contained in~$U$ as required.

    Finally, if the turbulation has partial flowboxes, then taking $\alpha_x$ and
    $\alpha_y$ to be~$\fu(x)$ and~$\fu(y)$ in Definition~\ref{defn:partial-flowbox}
    immediately yields the result.
\end{proof}

\begin{rem}
    It is enough in Lemma~\ref{lem:pos-cond}(b) for each tartan to be
    either stable bi-Lipschitz regular or unstable bi-Lipschitz regular.
    Likewise, it is enough in~(c) for just one of the turbulations to have partial
    flowboxes.
\end{rem}

\medskip\medskip

We finish this section with a first consequence of fullness.

\begin{lem}
    Let $(\cT^s, \nu^s), (\cT^u, \nu^u)$ be measurable pseudo-Anosov turbulations
    on~$\Sigma$, and let $\ell_s$ and~$\ell_u$ be stable and unstable streamlines.
    Then $\ell_s\cap\ell_u$ is dense in~$\Sigma$.
\end{lem}

\begin{proof}
    Suppose for a contradiction that there is a non-empty open subset~$U$
    of~$\Sigma$ with $\ell_s\cap\ell_u\cap U = \emptyset$.

    By fullness there is a positive measure compatible tartan~$R$ contained in~$U$.
    Any two stable and any two unstable fibers of~$R$ contain stable arcs $s_1,
        s_2$ and unstable arcs $u_1, u_2$ which bound an open disk~$V\subset U$. There
    are infinitely many such disks: choose one which does not contain an entire end
    of either $\ell_s$ or $\ell_u$ (recall that the streamlines are dense, but not
    necessarily bi-dense).

    Since $\ell_s$ is dense in~$\Sigma$, it contains a point~$x$ of~$V$. By the
    choice of~$V$, the closure of the path component of $\ell_s\cap V$
    containing~$x$ is an arc~$\alpha$ with endpoints in $u_1\cup u_2$. The
    endpoints of~$\alpha$ can't both lie on the same~$u_i$, since then $\alpha$ and
    $u_i$ would bound a disk which $\ell_u$ cannot enter, contradicting the
    denseness of~$\ell_u$. We conclude that there is an arc $\alpha\subset
        \overline{V}$ of $\ell_s$ with endpoints on $u_1$ and $u_2$.

    Analogously, there is an arc $\beta \subset\overline{V}$ of $\ell_u$ with
    endpoints on $s_1$ and $s_2$. Then $\alpha$ and $\beta$ intersect in $V\subset
        U$, which is the required contradiction.
\end{proof}

\section{Topological dynamics of measurable pseudo-Anosov maps}

In this section we prove our main results about the topological dynamics of
measurable pseudo-Anosov maps: that they are topologically transitive
(Theorem~\ref{thm:top-trans}) and that they have dense periodic points
(Theorem~\ref{thm:dense-per}). It follows~\cite{BBCDS} that they also exhibit
sensitive dependence on initial conditions, and are therefore chaotic in the sense
of Devaney~\cite{Devaney}.

\begin{thm} \label{thm:top-trans}
    Measurable pseudo-Anosov maps are topologically transitive.
\end{thm}

\begin{proof}
    (See Figure~\ref{fig:trans}.) Let $F\colon\Sigma\to\Sigma$ be measurable
    pseudo-Anosov with dilatation $\lambda>1$, and let $U$ and~$V$ be non-empty
    open subsets of~$\Sigma$. We need to show that there is some~$n>0$ with
    $F^n(U)\cap V \not=\emptyset$.

    By fullness, there is a positive measure compatible tartan~$R$ contained
    in~$U$. We write $S=R([0, 1], [1/3, 2/3])$, so that $w^s(S)=w^s(R)$ and
    $w^u(S)=w^u(R)/3$, and hence $\mu(S^\pitchfork)=\mu(R^\pitchfork)/3 > 0$.

    Applying the Poincar\'e recurrence theorem to~$F^{-1}$ yields a point $x_0\in
        S^\pitchfork$ and a sequence $n_i\to\infty$ such that $x_i=F^{-n_i}(x_0)\in
        S^\pitchfork$ for all~$i$.

    Since the unstable streamline~$\ell$ through~$x_0$ is dense in~$\Sigma$, it contains some point~$y\in V$. Write $K=\nu^u_\ell([x_0, y]_\ell)$, and pick~$i$
    large enough that $\lambda^{n_i} w^u(R) / 3 > K$.

    Let~$\fu \subset U$ be the unstable fiber of~$R$ which contains~$x_i$. Then
    $x_0\in F^{n_i}(\fu)$, and each of the two components of
    $F^{n_i}(\fu)\setminus\{x_0\}$ has stream measure at least $\lambda^{n_i}
        w^u(R) / 3 > K$. Hence $y\in F^{n_i}(\fu) \subset F^{n_i}(U)$, so that
    $F^{n_i}(U)\cap V\not=\emptyset$ as required.
\end{proof}

\begin{figure}[htbp]
    \begin{center}
        \includegraphics[width=0.8\textwidth]{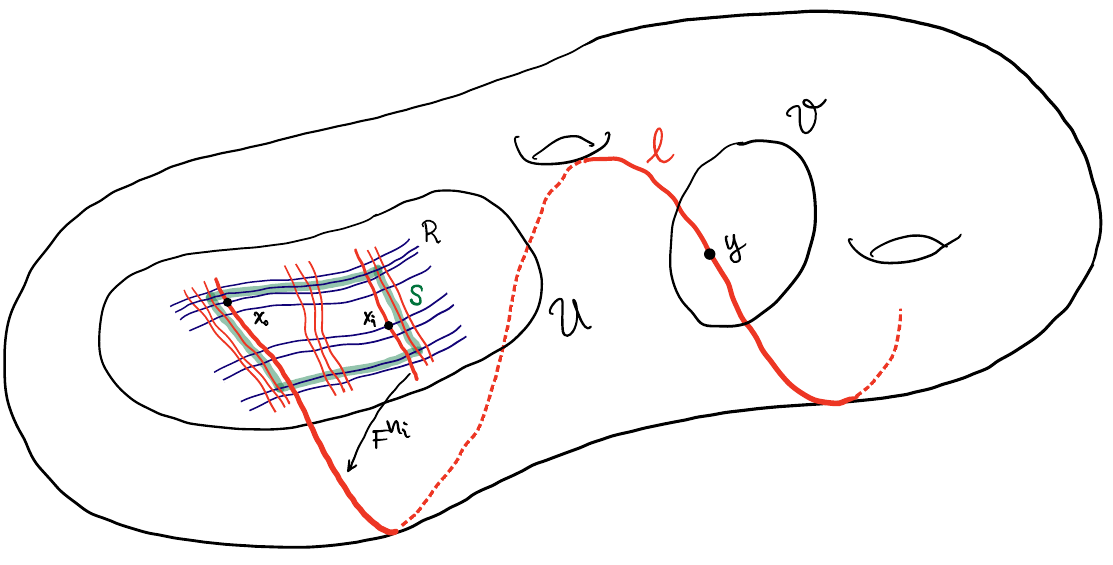}
    \end{center}
    \caption{Proof of topological transitivity}
    \label{fig:trans}
\end{figure}

The proof of density of periodic points is more challenging, and makes use of the
density result of Lemma~\ref{lem:density_intervals}.

\begin{thm} \label{thm:dense-per}
    Measurable pseudo-Anosovs have dense periodic points.
\end{thm}

\begin{proof}
    Let $F\colon\Sigma\to\Sigma$ be measurable pseudo-Anosov with
    dilatation~$\lambda$, and let~$U$ be a non-empty open subset of~$\Sigma$. We
    need to show that there is a periodic point of~$F$ in~$U$.

    By fullness, there is a positive measure compatible oriented tartan~$R$
    contained in~$U$. We will show that there is a positive integer~$n$ and an open
    rectangle~$V\subset U$ such that (see Figure~\ref{fig:rect}):
    \begin{enumerate}[(V1)]
        \item $\partial V$ is the union of two arcs $\alpha$ and $\alpha'$
              contained in stable fibers of~$R$, and two arcs $\beta$ and $\beta'$
              contained in unstable fibers of~$R$;

        \item the $F^n$-images of these boundary arcs are also contained in stable
              and unstable fibers of~$R$; and

        \item $F^n(\beta)$ and $F^n(\beta')$ both intersect both $\alpha$ and
              $\alpha'$.
    \end{enumerate}

    It follows that the fixed point index $\ind(F^n, V)$ of~$V$ under~$F^n$ is either
    $+1$ or $-1$ (depending on whether or not the orientations of the stable boundary
    arcs, and the unstable boundary arcs, are preserved by $F^n$), so that there is a
    fixed point of~$F^n$ in $V\subset U$ as required.

    Note that, if~(V1) holds, then it is immediate from the definition of measurable
    pseudo-Anosov maps that the $F^n$-images of $\alpha$, $\alpha'$, $\beta$, and
    $\beta'$ are stream arcs, but that this is not enough: we need these images to be
    contained in fibers of~$R$ so that we can control how they lie with respect to the
    original arcs.

    Define subtartans $C$, $W$, $E$, $S$, and $N$ of~$R$ (see Figure~\ref{fig:rect}) by
    \begin{eqnarray*}
        C &=& R\,([1/3, 2/3], [1/3, 2/3]), \\
        W &=& R\,([0, 1/3], [1/3, 2/3]), \\
        E &=& R\,([2/3, 1], [1/3, 2/3]), \\
        S &=& R\,([1/3, 2/3], [0, 1/3]), \text{ and}\\
        N &=& R\,([1/3, 2/3], [2/3, 1]),
    \end{eqnarray*}
    the Central, West, East, South, and North subtartans of $1/3$ the stable and
    unstable widths of~$R$, where the stable streamlines of~$R$ run from West to East
    and the unstable streamlines from South to North. We will construct the
    rectangle~$V$ so that the unstable boundary arcs $\beta$ and $\beta'$ lie in $W$
    and $E$, while the image stable boundary arcs $F^n(\alpha)$ and $F^n(\alpha')$ lie
    in $S$ and $N$ (not necessarily respectively in each case): this will ensure
    that~(V3) holds.

    \begin{figure}[htbp]
        \begin{center}
            \includegraphics[width=0.6\textwidth]{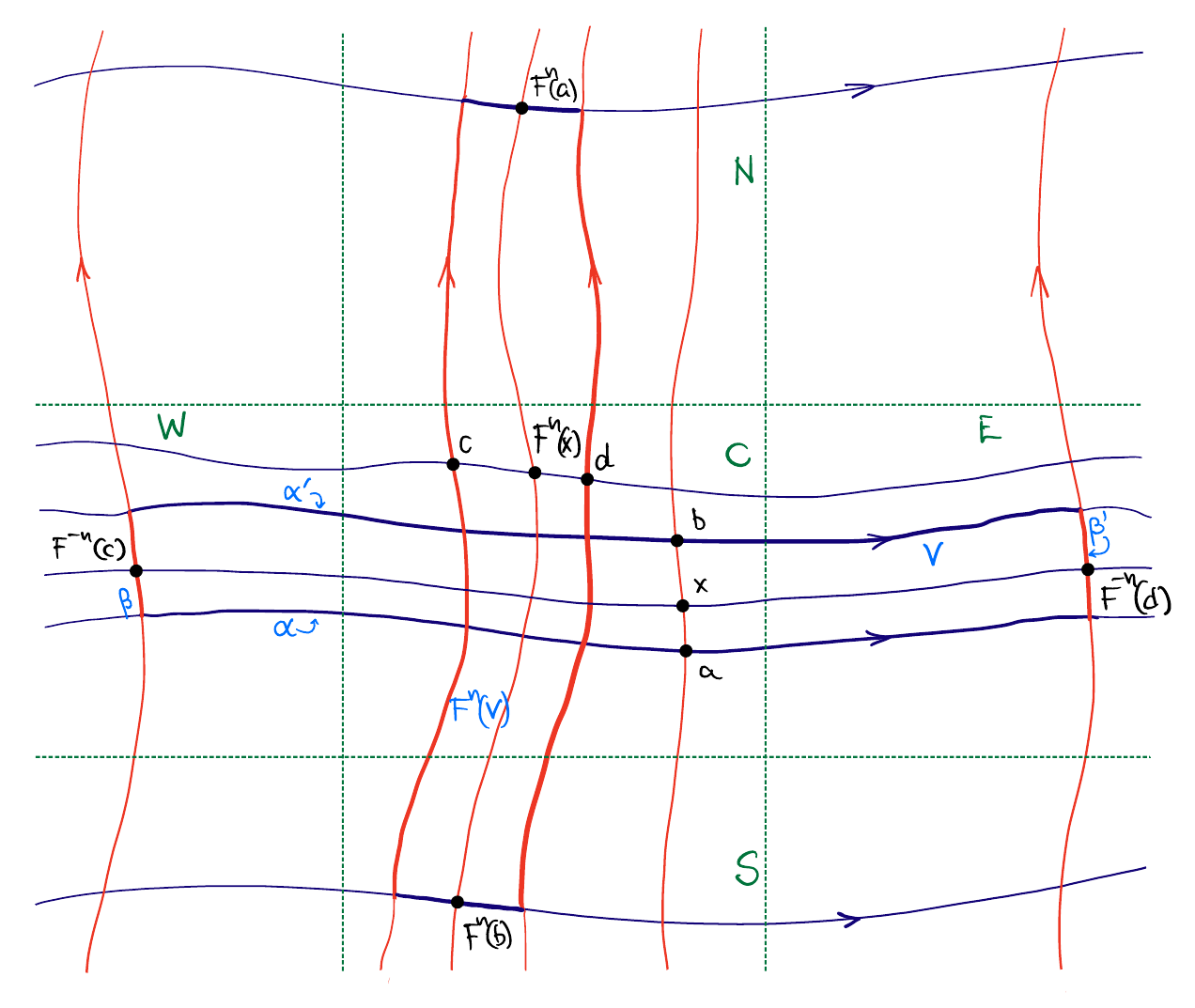}
        \end{center}
        \caption{The rectangle~$V$ and its image under $F^n$; the subdivision
            of~$R$; the points~$x$ and $F^n(x)$ and the construction of $a$, $b$, $c$, and $d$}
        \label{fig:rect}
    \end{figure}

    Let
    \begin{compactitem}
        \item $L$ be an upper bound on the stream measures of the fibers of~$R$
        (which exists by Definitions~\ref{defn:tartan}(c));

        \item $\delta\in(0,1)$ be such that $\delta < \min(w^s(R), w^u(R))/4\lambda L$;

        \item $\eta_0\in (0, L)$ be such that $\mu(X_{\delta, \eta_0}(C)) > 0$
        (see Definition~\ref{def:density_intervals} and
        Lemma~\ref{lem:density_intervals}); and

        \item $n_0>0$ be such that $\eta_0\in[\lambda^{-n_0}L, \lambda^{-n_0+1}L)$.
    \end{compactitem}

    By the Poincar\'e recurrence theorem applied to $X_{\delta, \eta_0}(C)$, there is
    an integer $n>n_0$ and a point $x\in X_{\delta, \eta_0}(C)$ such that $F^n(x)\in
        X_{\delta, \eta_0}(C)$.

    Pick $\eta\in (\lambda^{-n}L, \lambda^{-n+1}L)$. Since $\eta_0\ge \lambda^{-n_0}L$
    and $n>n_0$ we have $\eta < \eta_0$, so that $X_{\delta, \eta}(C) \supset
        X_{\delta, \eta_0}(C)$, and in particular $\{x, F^n(x)\} \subset X_{\delta,
                \eta}(C)$.

    Note that $X_{\delta, \eta}(C)\subset C^\pitchfork$, so that $x$ and $F^n(x)$ are
    contained in both stable and unstable fibers of~$C$, and hence in stable and
    unstable fibers of~$R$.

    We now use the point~$x$ to build the rectangle~$V$, starting with its stable
    boundary arcs~$\alpha$ and~$\alpha'$: we will find points $a$ and $b$ of the
    unstable streamline of~$x$, on either side of~$x$ and within stream measure~$\eta$
    of~$x$, which will lie on these arcs. The key properties are
    \begin{compactitem}
        \item that $a, b \in C^\pitchfork \subset R^\pitchfork$, so that they lie on
        stable fibers of~$R$; and
        \item that $F^n(a)$ and $F^n(b)$ lie in $S^\pitchfork \subset R^\pitchfork$ and
        $N^\pitchfork \subset R^\pitchfork$ (not necessarily respectively), so that (a)
        they also lie on stable fibers of~$R$, and (b) $F^n(s)$ and $F^n(t)$ lie in $S$
        and $N$, so that the image rectangle $F^n(V)$ stretches across~$C$ in the
        unstable direction.
    \end{compactitem}

    Let $\fu$ be the unstable fiber of~$R$ through $F^n(x)$. We have
    \mbox{$\nu^u_{F^n(x)}(\fu)\le L$} by definition of~$L$: since also $\lambda^{-n}L <
        \eta$, we have $F^{-n}(\fu) \subset (x-\eta, x+\eta)_u$.

    We first find the point $a\in(x-\eta, x)_u$. By definition of~$X_{\eta,
                \delta}(C)$, we have
    \begin{equation}\label{eq:1}
        \nu^u_x( (x-\eta, x)_u \cap C^\pitchfork) \ge (1-\delta)\eta.
    \end{equation}
    We will assume that $F^{-n}$ sends $\fu$ into the unstable fiber through~$x$
    preserving the fiber orientation, so that $F^{-n}(S^s \cap \fu) \subset (x-\eta,
        x)_u$: if it reverses fiber orientation, then we need only replace~$S$ with~$N$
    (Figure~\ref{fig:rect} depicts the orientation-reversing case). Now
    \begin{equation}\label{eq:2}
        \nu^u_x(F^{-n}(S^s\cap \fu)) = \lambda^{-n}\nu^u_{F^n(x)}(S^s \cap \fu)
        = \lambda^{-n}w^u(S) > \lambda^{-n}w^u(R)/4 > \delta\eta,
    \end{equation}
    where the final inequality comes from $\delta < w^u(R)/4\lambda L$ and $\eta <
        \lambda^{-n+1}L$.

    It follows from~\eqref{eq:1} and~\eqref{eq:2} that there is some $a\in(x-\eta,
        x)_u$ which lies in both $C^\pitchfork$ and in $F^{-n}(S^s\cap\fu)$. That is, $a\in C^\pitchfork$ and $F^n(a)\in S^\pitchfork \supset S^s\cap \fu$ as required.

    Applying the same argument to $(x, x+\eta)_u$ yields a point $b\in C^\pitchfork$
    with $F^n(b)\in N^\pitchfork$ (or $S^\pitchfork$ if $F^{-n}$ reverses fiber
    orientation).

    An exactly analogous argument applied to $(F^n(x)-\eta, F^n(x)+\eta)_s$ yields
    points $c, d\in C^\pitchfork \subset R^\pitchfork$ such that $F^{-n}(c)$ and
    $F^{-n}(d)$ lie in $W^\pitchfork\subset R^\pitchfork$ and $E^\pitchfork\subset R^\pitchfork$ (not necessarily respectively).

    The rectangle~$V$ whose boundary is contained in the stable fibers of~$R$ through
    $a$ and $b$, and the unstable fibers of~$R$ through $F^{-n}(c)$ and $F^{-n}(d)$
    then has the required properties.

\end{proof}

\section{Ergodicity} \label{sec:ergodicity}

In this section we use a variant of the Hopf argument (see, for example,
\cite{brin, burnsetal, coudene}) to prove ergodicity of measurable pseudo-Anosov
maps. In classical uses of this argument, such as to prove the ergodicity of volume
preserving Anosov diffeomorphisms, the main effort is in proving that the invariant
foliations are absolutely continuous. This makes it possible to use a Fubini
theorem to show that Birkhoff averages of $L^1$ functions are constant a.e.\ in
charts. Since the manifold is connected and each point is contained in an open
chart, this implies that the global Birkhoff average is constant a.e., yielding
ergodicity.

In the case of measurable pseudo-Anosov maps, the compatibility conditions on a
tartan~$R$ allow a Fubini argument to show that Birkhoff averages are constant
a.e.\ in $R^\pitchfork$. However, since tartans are not open sets we need an
additional hypothesis connecting tartans in order to deal with the global averages.

Let $F\colon\Sigma\to\Sigma$ be measurable pseudo-Anosov with
dilatation~$\lambda>1$, and let $\phi\colon\Sigma\to\R$ be continuous. By
Birkhoff's ergodic theorem, the limits
\begin{eqnarray*}
    \phi^+(x) &=&
    \lim_{n\to\infty} \frac{1}{n} \sum_{i=0}^{n-1}\phi(F^i(x))
    \text{ and }\\
    \phi^-(x) &=&
    \lim_{n\to\infty} \frac{1}{n} \sum_{i=0}^{n-1}\phi(F^{-i}(x))
\end{eqnarray*}
exist and are equal $\mu$-almost everywhere.

Since continuous functions are dense in $L^1$, using the converse of Birkhoff's
ergodic theorem, it suffices for ergodicity to show that for every continuous
$\phi$, the functions $\phi^+$ and $\phi^-$ are constant almost everywhere.

Throughout this section we consider a fixed such function~$\phi$.

\begin{lem} \label{lem:stream-arc-phi}
    Suppose that $x$ and $y$ lie in the same stable (respectively unstable)
    streamline~$\ell$, and that $\phi^+(x)$ (respectively $\phi^-(x)$) exists. Then
    $\phi^+(y)$ (respectively $\phi^-(y)$) exists and is equal to $\phi^+(x)$
    (respectively $\phi^-(x)$).
\end{lem}
\begin{proof}
    If $x$ and $y$ lie in the same stable streamline~$\ell$ with $\nu_\ell([x,y]_s)
        = K$, then we have $\nu_{F^i(\ell)}([F^i(x), F^i(y)]_s) = K/\lambda^i$.

    Given $\epsilon>0$, let $N_\epsilon=\{(a,b)\in\Sigma\times\Sigma\,:\,
        |\phi(a)-\phi(b)| < \epsilon\}$, a neighborhood of the diagonal. By tameness
    there is a $\delta>0$ such that $(a,b)\in N_\epsilon$ whenever the stream
    measure of $[a,b]_s$ is less than~$\delta$: therefore $|\phi(F^i(x)) -
        \phi(F^i(y))| < \epsilon$ for all sufficiently large~$i$. It follows that
    $\phi^+(y)=\phi^+(x)$ as required. The argument when $x$ and $y$ are on the
    same unstable streamline is analogous.
\end{proof}

\begin{defns}[$Z$, $\sim$]
    Let
    \[
        Z = \{x\in \Sigma\,:\, \phi^+(x) \text{ and } \phi^-(x)
        \text{ exist and are equal}  \},
    \]
    so that $\mu(Z)=1$ by Birkhoff's ergodic theorem as noted above.

    Let $x, y\in Z$. We write $x\sim y$ if there is a finite collection
    $\{\alpha_1, \dots, \alpha_k\}$ of stream arcs with $x\in\alpha_1$ and
    $y\in\alpha_k$, such that $\alpha_{i}\cap\alpha_{i+1}$ contains a point
    $z_i\in Z$ for $1\le i\le k - 1$.
\end{defns}

$\sim$ is an equivalence relation on~$Z$, and $\phi^+$ and $\phi^-$ are
constant on equivalence classes. For, writing $z_0=x$ and $z_k=y$, we have for
each $1\le i\le k$ that either $\phi^+(z_{i-1})=\phi^+(z_i)$ or
$\phi^-(z_{i-1})=\phi^-(z_i)$ by Lemma~\ref{lem:stream-arc-phi}; and hence
$\phi^\pm(z_{i-1})=\phi^\pm(z_i)$ since $z_{i-1}, z_i\in Z$.

It follows that $F$ is ergodic if $\sim$ has a full measure equivalence class.

\begin{lem} \label{lem:const_on_tartan}
    Let~$R$ be a positive measure compatible tartan. Then $\phi^+=\phi^-$ is
    constant almost everywhere in~$R^\pitchfork$.
\end{lem}

\begin{proof}
    Let $\fs$ and $\fu$ be any stable and unstable fibers of~$R$, and let $V =
        R^\pitchfork \cap Z$, a full measure subset of $R^\pitchfork$.

    By~\eqref{eq:int_formula},
    \[
        \int_{E^\fs} \nu_\fu(E^\fu)\,d\nu_\fs(x) = \mu(R^\pitchfork) = \mu(V) =
        \int_{E^\fs} \nu_{\fu}(V^\fu(x))\,d\nu_\fs(x),
    \]
    so that $W=\{x\in E^\fs\,:\, \nu_\fu(V^\fu(x)) = \nu_\fu(E^\fu)\}$ has full
    $\nu_\fs$-measure in~$E^\fs$. Let $S$ be the subtartan of~$R$ obtained by
    discarding the unstable fibers which don't intersect~$W$. Then
    $\mu(S^\pitchfork) = \mu(R^\pitchfork)$, and $\phi^+=\phi^-$ is constant on
    $S^\pitchfork \cap Z$, a full measure subset of $R^\pitchfork$. For if $a,
        a'\in S^\pitchfork\cap Z$, then let $x, x'\in W$ be such that $a\in \fu(x)$ and
    $a'\in\fu(x')$, and let $y\in V^\fu(x)\cap V^\fu(x')$. Then the stream arcs
    $[a, \fu(x)\pitchfork\fs(y)]_u$, $[\fu(x)\pitchfork\fs(y),
                \fu(x')\pitchfork\fs(y)]_s$, and $[\fu(x')\pitchfork\fs(y), a']_u$ connect $a$
    and $a'$ and intersect at points of $V\subset Z$, so that $a\sim a'$ as
    required (see Figure~\ref{fig:phi_stream}).
\end{proof}
\begin{figure}[htbp]
    \begin{center}
        \includegraphics[width=0.5\textwidth]{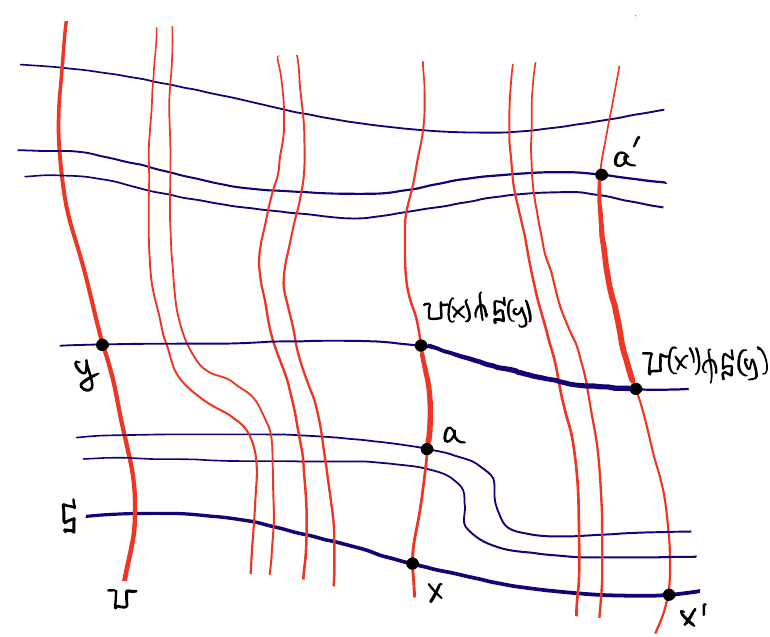}
    \end{center}
    \caption{Stream arcs showing that $a\sim a'$}
    \label{fig:phi_stream}
\end{figure}

In order to establish ergodicity, we need some means of relating the Birkhoff
averages of $\phi$ on different positive measure tartans. The following is a
natural, though strong, condition:

\begin{defns}[Tartan connection hypothesis] \label{def:sat}

    We say that the measurable pseudo-Anosov turbulations $(\cT^s, \nu^s)$,
    $(\cT^u, \nu^u)$ satisfy the \emph{tartan connection hypothesis} if for any
    compatible positive measure tartans~$R$ and~$S$, the stable streamlines
    containing fibers of~$R$ intersect the unstable streamlines containing fibers
    of~$S$ in positive measure.

\end{defns}

\begin{thm} \label{thm:ergodic}
    If the turbulations of a measurable pseudo-Anosov map~$F$ satisfy the tartan
    connection hypothesis, then~$F$ is ergodic.
\end{thm}
\begin{proof}

    Let~$R_i$ be a countable collection of positive measure compatible tartans with
    $\mu(\bigcup_i R_i^\pitchfork)=1$.

    By Lemma~\ref{lem:const_on_tartan} there are constants $\phi_i$ such that
    $\phi^+=\phi^-=\phi_i$ almost everywhere in $R_i^\pitchfork$, and it suffices
    to show that $\phi_i=\phi_j$ for all $i$ and $j$.

    Given~$i$ and~$j$, let $S_i$ and $S_j$ be the positive measure compatible
    subtartans of~$R_i$ and~$R_j$ obtained by discarding the stable fibers of~$R_i$
    on which $\phi^+\not=\phi_i$ (or~$\phi^+$ does not exist), and the unstable
    fibers of~$R_j$ on which $\phi^-\not=\phi_j$ (or~$\phi^-$ does not exist).

    By the tartan connection hypothesis, the intersection~$A$ of the union of the
    stable streamlines containing fibers of~$S_i$ with the union of the unstable
    streamlines containing fibers of~$S_j$ has positive measure, so that $\mu(A\cap
        Z)>0$. However if $a\in A\cap Z$ then $\phi^+(a)=\phi_i$ and $\phi^-(a)=\phi_j$
    by Lemma~\ref{lem:stream-arc-phi}. Since $a\in Z$ we have $\phi_i=\phi_j$ as
    required.

\end{proof}

Clearly Theorem~\ref{thm:ergodic} is also true under the weaker condition that for
any compatible positive measure tartans~$R$ and~$S$, there is a chain $R=T_0, T_1,
    \dots, T_k=S$ of compatible tartans such that, for each~$i$, the stable streamlines
containing fibers of~$T_i$ intersect the unstable streamlines containing fibers
of~$T_{i+1}$ in positive measure. This is worth mentioning because this weaker
condition is in turn implied by the \emph{Tartan chain hypothesis}: that for any
compatible positive measure tartans~$R$ and~$S$, there is a chain as above such
that $\mu(T_i^\pitchfork \cap T_{i+1}^\pitchfork)>0$ for each~$i$.

\bibliographystyle{amsplain}
\bibliography{dynmpa}

\end{document}